\documentclass[12pt, reqno]{amsart}

\author[P.~Leonetti]{Paolo Leonetti}
\address{Department of Economics, Universit\`a degli Studi dell'Insubria, via Monte Generoso 71, 21100 Varese, Italy}
\email{leonetti.paolo@gmail.com}



\title{On maldistributed sequences and meager ideals}

\usepackage[T1]{fontenc}
\usepackage{amsmath}
\usepackage{amssymb}
\usepackage{amsthm}
\usepackage[left=3.5cm, right=3.5cm]{geometry} 
\usepackage{hyperref}
\usepackage{fancyhdr}
\usepackage{enumitem}
\usepackage{comment}
\usepackage{nicefrac}
\usepackage{comment}
\usepackage{bm}
\usepackage{mathrsfs}
\usepackage{graphicx}
\usepackage[utf8]{inputenc}
\usepackage{cancel}
\usepackage{mathtools}

\usepackage{tikz}
\usetikzlibrary{decorations.pathreplacing,matrix,arrows,positioning,automata,shapes,shadows,calc,fadings,decorations,snakes,through,intersections}

\AtBeginDocument{%
   \def\MR#1{}
}

\newtheorem{thm}{Theorem}[section]

\newtheorem{prop}[thm]{Proposition}

\theoremstyle{definition} 
\newtheorem{defi}[thm]{Definition}
\let\olddefi\defi
\renewcommand{\defi}{\olddefi\normalfont}
\newtheorem{question}{Question}
\let\oldquestion\question
\renewcommand{\question}{\oldquestion\normalfont}
\newtheorem{example}[thm]{Example}
\let\oldexample\example
\renewcommand{\example}{\oldexample\normalfont}
\newtheorem{rmk}[thm]{Remark}
\let\oldrmk\rmk
\renewcommand{\rmk}{\oldrmk\normalfont}

\hypersetup{
    pdftitle={},
    pdfauthor={},
    pdfmenubar=false,
    pdffitwindow=true,
    pdfstartview=FitH,
    colorlinks=true,
    linkcolor=blue,
    citecolor=green,
    urlcolor=cyan
}

\providecommand{\MR}[1]{}

\providecommand{\MR}{\relax\ifhmode\unskip\space\fi MR }

\begin{document}

\maketitle
\thispagestyle{empty}

\begin{abstract} 
We show that an ideal $\mathcal{I}$ on $\omega$ is meager if and only if 
the set of sequences $(x_n)$ taking values in a 
Polish space $X$ 
for which all elements of $X$ are $\mathcal{I}$-cluster points of $(x_n)$ is comeager. 
The latter condition is also known as $\nu$-maldistribution, where $\nu: \mathcal{P}(\omega)\to \mathbb{R}$ is the $\{0,1\}$-valued submeasure defined by $\nu(A)=1$ if and only if $A\notin \mathcal{I}$.  
It turns out that the meagerness of $\mathcal{I}$ is also equivalent to 
a technical condition given by Mi{\v s}{\'i}k and T{\'o}th in [J.~Math.~Anal.~Appl.~\textbf{541} (2025), 128667]. 
Lastly, we show that the analogue of the first part holds replacing $\nu$ with $\|\cdot\|_\varphi$, where $\varphi$ is a lower semicontinuous submeasure. 
\end{abstract}


\section{Introduction}\label{sec:intro}

Let $\omega$ stands for the set of nonnegative integers. We say that a map $\nu: \mathcal{P}(\omega) \to \mathbb{R}$ is a \emph{diffuse capacity} if it is a monotone map (that is, $\nu(A)\le \nu(B)$ for all $A\subseteq B\subseteq \omega$) such that $\nu(F)=0$ and $\nu(\omega\setminus F)=1$ for all finite $F\subseteq \omega$. If, in addition, $\nu$ is subadditive (that is, $\nu(A\cup B)\le \nu(A)+\nu(B)$ for all $A,B\subseteq \omega$) then we call it \emph{diffuse submeasure}. 
Given a topological space $X$, a sequence $\bm{x}=(x_n: n \in \omega) \in X^\omega$ is called $\nu$\emph{-maldistributed} if 
$$
\nu(\{n \in \omega: x_n \in U\})=1
$$
for all nonempty open sets $U\subseteq X$, cf. \cite[Definition 3.1]{MR4781067} for the case of separable metric spaces. 
Moreover, define the set
$$
\Sigma_\nu(X):=\{\bm{x} \in X^\omega: \bm{x} \text{ is }\nu\text{-maldistributed}\,\}
$$
and endow $X^\omega$ with the product topology. 
It is worth noting, as a particular case, that if $\nu$ is the diffuse submeasure defined by $\nu(S):=0$ if $S\subseteq \omega$ is finite and $\nu(S):=1$ otherwise then a continuous map $T: X\to X$ is commonly known as hypercyclic if and only if there exists $x_0 \in X$ such that its orbit $(T^nx_0: n \in \omega)$ is $\nu$-maldistributed, cf. \cite{MR3552249, LeoDynamical} and references therein.  

Very recently, Mi{\v s}{\'i}k and T{\'o}th proved a sufficient technical condition to guarantee that, from a topological viewpoint, most sequences with values in $X$ are $\nu$-maldistributed, namely, the complement of $\Sigma_\nu(X)$ is meager (hence, contained in a countable union of closed sets with empty interior). With the above premises, their main result \cite[Theorem 3.1]{MR4781067} can be formulated as follows: 
\begin{thm}\label{thm:misiktothmain}
    Let $X$ be a separable metric space and suppose that $\nu: \mathcal{P}(\omega) \to \mathbb{R}$ is a diffuse capacity which satisfies the condition\textup{:}
    \begin{equation}\label{eq:misikcondition}
    \begin{split}
\forall \alpha \in (0,1), 
&\,\exists g_\alpha \in \omega^\omega, 
\,\forall A\subseteq \omega: \\
&\hspace{-10mm}\nu(\omega\setminus A)\le 1-\alpha \implies 
\exists n_{\alpha,A} \in \omega, \forall n \ge n_{\alpha,A}: A\cap [n,n+g_\alpha(n)] \neq \emptyset.
\end{split}
    \end{equation}
Then 
$\Sigma_\nu(X)$ 
is comeager. 
\end{thm}

Results on the same spirit of Theorem \ref{thm:misiktothmain} can be found also in \cite
{MR4600193, MR4566746, MR4506089, MR2075041}. 
In the same work, Mi{\v s}{\'i}k and T{\'o}th asked whether the converse of Theorem \ref{thm:misiktothmain} holds, namely, whether there exists a separable metric space $X$ and a diffuse capacity $\nu$ such that $\Sigma_\nu(X)$ is comeager, while condition \eqref{eq:misikcondition} does \emph{not} hold, see \cite[Open Problem 5.1]{MR4781067}. 
Our aim is to answer it in the negative for a certain family of diffuse submeasures.

To this aim, recall that an ideal $\mathcal{I}\subseteq \mathcal{P}(\omega)$ is a family of subsets stable under finite unions and subsets. Moreover, it is assumed that the family of finite sets $\mathrm{Fin}$ is contained in $\mathcal{I}$, and that $\omega\notin \mathcal{I}$. 
The dual filter of an ideal $\mathcal{I}$ is denoted by $\mathcal{I}^\star:=\{S\subseteq \omega: \omega\setminus A\in \mathcal{I}\}$. Identifying $\mathcal{P}(\omega)$ with the Cantor space $2^\omega$, we can speak about the topological complexity of ideals (in particular, it makes sense to speak about meager ideals). Lastly, given a sequence $\bm{x}$ taking values in a topological space $X$, we say that $\eta \in X$ is an $\mathcal{I}$-cluster point of the sequence $\bm{x}$ if $\{n \in \omega: x_n \in U\} \in \mathcal{I}^+$ for all open neighborhoods $U$ of $\eta$, where $\mathcal{I}^+:=\mathcal{P}(\omega)\setminus \mathcal{I}$ stands for the family of $\mathcal{I}$-positive sets. 
We refer to \cite{MR3920799} for basic facts and characterizations of the set of $\mathcal{I}$-cluster points, which is denoted by $\Gamma_{\bm{x}}(\mathcal{I})$.


Our main result shows that condition \eqref{eq:misikcondition} for the submeasure $\nu:=\bm{1}_{\mathcal{I}^+}$ 
is equivalent to meagerness of the ideal $\mathcal{I}$, and also to the comeagerness of $\Sigma_\nu(X)$; this goes in the spirit of the characterizations of meagerness of $\mathcal{I}$ given in \cite{MR4566746}. 
\begin{thm}\label{thm:answeropenproblem}
Let $X$ be a Polish space with $|X|\ge 2$, let $\mathcal{I}$ be an ideal on $\omega$, and define the diffuse submeasure $\nu:=\bm{1}_{\mathcal{I}^+}$. Then the following are equivalent\textup{:}
\begin{enumerate}[label={\rm (\roman{*})}]
\item \label{item:1mainmeager} $\nu$ satisfies condition \eqref{eq:misikcondition}\textup{;}
\item \label{item:2mainmeager} $\mathcal{I}$ is meager\textup{;}
\item \label{item:3mainmeager} $\Sigma_\nu(X)$ is comeager\textup{.}
\end{enumerate}
\end{thm}

It is worth noting, as it 
follows by \cite[Proposition 2.11 and Theorem 6.2]{FKL}, that, if $\mathcal{I}$ is an ideal on $\omega$ and $X$ is an infinite separable metric space, then the existence of a $\bm{1}_{\mathcal{I}^+}$-maldistributed sequence is equivalent to the fact that for every $n \in \omega$ there exists a sequence $\bm{x} \in X^\omega$ 
with at least $n$ $\mathcal{I}$-cluster points, which happens if and only if $\mathcal{I}$ is not a Fubini sum of finitely many maximal ideals. 


Lastly, we provide a really large class of diffuse submeasures which satisfy condition \eqref{eq:misikcondition}. 
To this aim, a monotone subadditive map $\varphi: \mathcal{P}(\omega)\to [0,\infty]$ is said to be a \emph{lower semicontinuous submeasure} (in short, lscsm) if it satisfies $\varphi(F)<\infty$ for all finite $F \subseteq \omega$ and, in addition, 
$$
\forall A\subseteq \omega, \quad \varphi(A)=\sup\{\varphi(A\cap [0,n]): n \in \omega\}.
$$
Notice that the above property is precisely the lower semicontinuity of the submeasure $\varphi$, regarding its domain $\mathcal{P}(\omega)$ as the Cantor space $2^\omega$, that is, if $A_n \to A$ then $\liminf_n \varphi(A_n) \ge \varphi(A)$. Examples of lscsms include $\varphi(A)=|A|$ or $\varphi(A)=\sum_{n \in A}1/(n+1)$ or $\varphi(A)=\sup_{n\ge 1} |A\cap [0,n]|/n$, cf. also \cite[Chapter 1]{MR1711328}. 

Given a lscsm $\varphi: \mathcal{P}(\omega)\to [0,\infty]$, define the family
$$
\mathrm{Exh}(\varphi):=\{S\subseteq \omega: \|S\|_\varphi=0\},
\,\,\,\, \text{ where }\,\,\|S\|_\varphi:=\lim_{n\to \infty} \varphi(S\setminus [0,n]).
$$
Informally, $\|S\|_\varphi$ stands for the $\varphi$-mass at infinity of the set $S$. 
A classical result of Solecki \cite[Theorem 3.1]{MR1708146} states that an ideal $\mathcal{I}$ on $\omega$ is an analytic $P$-ideal if and only if there exists a lscsm $\varphi$ such that 
$$
\mathcal{I}=\mathrm{Exh}(\varphi)
\quad \text{ and }\quad 
\varphi(\omega)<\infty. 
$$
Here, $\mathcal{I}$ is said to be a $P$-ideal if for every sequence $(A_n)$ with values in $\mathcal{I}$ there exists $A \in \mathcal{I}$ such that $A_n\setminus A$ is finite for all $n \in \omega$. 

We remark that the family of analytic $P$-ideals is large and includes, among others, all Erd{\H o}s--Ulam ideals introduced by Just and Krawczyk in \cite{MR748847}, ideals generated by nonnegative regular matrices \cite{Filipow18, MR4041540}, 
the Fubini products $\emptyset \times \mathrm{Fin}$, which can be defined as 
$
\{A\subseteq \omega: \forall n \in \omega, A \cap I_n \in \mathrm{Fin}\},
$ 
where $(I_n)$ is a given partition of $\omega$ into infinite sets,  
certain ideals used by Louveau and Veli\v{c}kovi\'{c} \cite{Louveau1994}, 
and, more generally, density-like ideals and generalized density ideals \cite{MR3436368, MR4404626}. Additional pathological examples can be found in \cite{MR0593624}. 
It has been suggested in \cite{MR4124855, MR3436368} that the theory of analytic $P$-ideals may have some relevant yet unexploited potential for the study of the geometry of Banach spaces. 

\begin{prop}\label{prop:lscsm}
Let $X$ be a separable metric space 
and $\varphi: \mathcal{P}(\omega)\to [0,\infty]$ be a lscsm such that $\|\omega\|_\varphi=1$. 
Then $\|\cdot\|_\varphi$ is a diffuse submeasure which satisfies condition \eqref{eq:misikcondition}. Hence, $\Sigma_{\|\cdot\|_\varphi}(X)$ is comeager.  
\end{prop}

The above simple result provides a generalization of \cite[Proposition 4.1]{MR4781067}, which corresponds to the case of lscsm $\varphi$ generating a density ideal as in \cite[Section 1.13]{MR1711328}, which in turn extends the main results in \cite{MR4183195}. 

To conclude, one might be tempted to conjecture that, at least in the case of the submeasures $\bm{1}_{\mathcal{I}^+}$, where $\mathcal{I}$ is an ideal on $\omega$, the set $\Sigma_\nu(X)$ is either meager or comeager. 
For instance, if $X$ is a compact metric space with $|X|\ge 2$ and $\mathcal{I}$ is a maximal ideal on $\omega$ (that is, the complement of a free ultrafilter), then every sequence with values in $X$ would be $\mathcal{I}$-convergent, so that $\Sigma_\nu(X)=\emptyset$. 
However, 
the following example shows that this is not the case. 
\begin{example}\label{example:nomeagernocomeager}
    Endow $X:=\{0,1\}$ with the discrete topology, let $\mathcal{I}_0, \mathcal{I}_1$ be two maximal ideals on $2\omega$ and $2\omega+1$, respectively, and define 
    $$
    \mathcal{I}:=\{S\subseteq \omega: S\cap 2\omega \in \mathcal{I}_0 \text{ and } S\cap (2\omega+1) \in \mathcal{I}_1\}.
    $$
    Then $\Sigma_{\bm{1}_{\mathcal{I}^+}}(X)$ is neither meager nor comeager. In fact, for each $i,j \in X$ define 
    $$
    \mathcal{S}_{i,j}:=\{\bm{x} \in X^\omega: 
    \mathcal{I}_0\text{-}\lim \bm{x}\upharpoonright 2\omega=i \text{ and }
    \mathcal{I}_1\text{-}\lim \bm{x}\upharpoonright (2\omega+1)=j\}.
    $$
    Regarding $X$ as the Abelian group $\mathbb{Z}/2\mathbb{Z}$, it is easy to see that the above sets $\mathcal{S}_{i,j}$ are homeomorphic. Since $X^\omega$ is Polish and $\Sigma_{\bm{1}_{\mathcal{I}^+}}(X)=\mathcal{S}_{0,1} \cup \mathcal{S}_{1,0}$, we conclude that $\Sigma_{\bm{1}_{\mathcal{I}^+}}(X)$ is neither meager nor comeager. 
\end{example}

The proofs of our results are given in Section \ref{sec:proofs}. 


\section{Proofs} \label{sec:proofs}

Before the proofs of our main characterization, we start 
with the following intermediate result, cf. \cite[Theorem 3.1]{MR4566746}. This applies, in particular, to complete metric spaces $X$ with $|X|\ge 2$.
\begin{prop}\label{prop:converse}
    Let $X$ be a Hausdorff space with $|X|\ge 2$ and assume that $X^\omega$ is Baire. Suppose also that there exists $\eta \in X$ such that 
    \begin{equation}\label{eq:definitionSeta}
    \mathcal{S}_\eta:=\{\bm{x} \in X^\omega: \eta \in \Gamma_{\bm{x}}(\mathcal{I})\}
    \end{equation}
    is comeager. Then $\mathcal{I}$ is meager.
\end{prop}
\begin{proof}
Fix $\eta \in X$ such that $\mathcal{S}_\eta$ is comeager, and let $U,V\subseteq X$ be two disjoint nonempty open sets such that $\eta \in U$. 
Since $X^\omega$ is Baire, there exists a decreasing sequence $(G_n)$ of dense open subsets
of $X^\omega$ such that $\bigcap_n G_n$ is dense and contained in $\mathcal{S}_\eta$. 

Now, consider the following game defined by Laflamme in \cite{MR1367134}: Players I and II choose alternately subsets $C_0,F_0,C_1,F_1,\ldots$ of $\omega$, where the sets $C_0\supseteq C_1\supseteq \ldots$, which are chosen by Player I, are cofinite and the sets $F_k\subseteq C_k$, which are chosen by Player II, are finite. 
Player II is declared to be the winner if and only if $\bigcup\nolimits_k F_k \in \mathcal{I}^+$. We may suppose without loss of generality that $F_k\cap C_{k+1}=\emptyset$ and $C_k=[c_k,\infty)$  for all $k \in \omega$ (hence, the sequence $(c_{k})$ corresponds to arbitrary (large enough) choices made by Player I). By \cite[Theorem 2.12]{MR1367134}, Player II has a winning strategy if and only if $\mathcal{I}$ is meager. 
The remaining part of the proof consists in showing that Player II has a winning strategy. 

We will define recursively, together with the description of the strategy of Player II, also a decreasing sequence of basic open sets 
$$
A_0\supseteq B_0 \supseteq A_1\supseteq B_1 \supseteq \cdots
$$
in $X^\omega$ (recall that a basic open set in $X^\omega$ is a cylinder of the type $D=\{\bm{x} \in X^\omega: x_0\in W_0, x_1 \in W_1,\ldots,x_n \in W_n\}$ for some open sets $W_0,\ldots,W_n \subseteq X$, and we set $m(D):=n$).  
Suppose that the sets $C_0,F_0,\ldots,C_{k-1},F_{k-1},C_k\subseteq \omega$ have been already chosen and that the open sets $A_0, B_0, \ldots, A_{k-1}, B_{k-1} \subseteq X^\omega$ have already been defined, for some $k \in \omega$, where we assume by convention that $B_{-1}:=X^\omega$ and $m(X^\omega):=-1$. 
Then we define the sets $A_k, B_k$, and $F_k$ as follows: 
\begin{enumerate}[label={\rm (\roman{*})}]
\item $A_k:=\{\bm{x} \in B_{k-1}: x_n \in V \text{ for all }n\text{ with }m(B_{k-1})<n<c_k\}$;
\item $B_k$ is a nonempty basic open set contained in $G_k \cap A_k$ (note that this is possible since $G_k$ is open dense and $A_k$ is nonempty open); in addition, for each $n \in [c_k,m(B_k)]$, let $W_n\subseteq X$ be the smallest nonempty open set such that if $\bm{x} \in B_k$ then $x_n \in W_n$ (equivalently, $W_n$ is the unique open subset of $X$ such that the projection of the cylinder $B_k$ at the $n$-th coordinate is precisely $x_n \in W_n$). Replacing each $W_n$ with the smaller open set $W_n \cap U$ if $W_n\cap U\neq \emptyset$, it is possible to assume without loss of generality that either $W_n\subseteq U$ or $W_n \cap U=\emptyset$ for all $n \in [c_k,m(B_k)]$. 
\item $F_k:=\{n \in [c_k, m(B_k)]: W_n\subseteq U\}$ (note that this is a finite set, possibly empty). 
\end{enumerate} 
We obtain by construction that there exists a sequence $\bm{x}=(x_n: n \in\omega) \in X^\omega$ such that $\bm{x} \in \bigcap\nolimits_k B_k \subseteq \bigcap\nolimits_k G_k\subseteq \mathcal{S}_\eta$. 
This implies that $\eta$ is an $\mathcal{I}$-cluster point of $\bm{x}$, hence $\{n \in \omega: x_n \in U\} \in \mathcal{I}^+$. 
At the same time, by the definitions above  
$$
\{n \in \omega: x_n \in U\} 
=\bigcup\nolimits_k \{n \in [c_k, m(B_k)]: W_n\subseteq U\}
=\bigcup\nolimits_k F_k.
$$
This proves that Player II has a winning strategy. Therefore $\mathcal{I}$ is meager.
\end{proof}

\medskip

Note that Theorem \ref{thm:misiktothmain} proves, in particular, the implication \ref{item:1mainmeager} $\implies$ \ref{item:3mainmeager} of Theorem \ref{thm:answeropenproblem}. However, we provide below a self-contained proof. 
\begin{proof}[Proof of Theorem \ref{thm:answeropenproblem}] 
First of all, it is routine to check that, if $\mathcal{I}$ is an ideal on $\omega$, then $\nu:=\bm{1}_{\mathcal{I}^+}$ is a diffuse submeasure, and that it satisfies condition \eqref{eq:misikcondition} if and only if: 
\begin{equation}\label{eq:misikcondition2}
\exists g \in \omega^\omega, 
\,\forall A\in \mathcal{I}^\star,\, \exists n_{A} \in \omega,\, \forall n \ge n_{A}:\,\,\,\, A\cap [n,n+g(n)] \neq \emptyset.
\end{equation} 
Moreover, by Talagrand's characterization \cite[Theorem 2.1]{Talagrand}, the meagerness of $\mathcal{I}$ is equivalent to the existence of 
    a sequence $(I_n: n \in \omega)$ of intervals of $\omega$ such that $\max I_n<\min I_{n+1}$ for all $n \in \omega$ and that $S\in \mathcal{I}^+$ whenever $I_k\subseteq S$ for infinitely many $k \in \omega$. 

\medskip

\ref{item:1mainmeager} $\implies$ \ref{item:2mainmeager}. Suppose that condition \eqref{eq:misikcondition} holds for $\bm{1}_{\mathcal{I}^+}$ or, equivalently, condition \eqref{eq:misikcondition2} is satisfied. Observe that the latter is equivalent to the existence of $g \in \omega^\omega$ such that if $S:=\omega\setminus A$ contains infinitely many $[n,n+g(n)]$, then $A\notin \mathcal{I}^\star$, i.e., $S\in \mathcal{I}^+$. At this point, define $I_0:=[0,g(0)]$ and $I_{n+1}:=[a_n, a_n+g(a_n)]$ where $a_n:=1+\max I_n$ for all $n \in \omega$. It follows that $S \in \mathcal{I}^+$ whenever $S$ contains infinitely many intervals $I_n$. Hence $\mathcal{I}$ is meager by Talagrand's characterization.

\medskip

\ref{item:2mainmeager} $\implies$ \ref{item:1mainmeager}. Pick a sequence of intervals $(I_n)$ as in Talangrand's characterization, and define $g(n):=\max I_k$ where $k$ is the smallest nonnegative integer with $\min I_k\ge n$. Now, pick $A\subseteq \omega$ and suppose that there exists infinitely many $n \in \omega$ such that $A\cap [n,n+g(n)]=\emptyset$. Then $\omega\setminus A$ contains inifinitely many $I_n$, so that it belongs to $\mathcal{I}^+$. Therefore $A\notin \mathcal{I}^\star$, and condition \eqref{eq:misikcondition2} holds.

\bigskip

\ref{item:2mainmeager} $\implies$ \ref{item:3mainmeager}. Pick a sequence of intervals $(I_n)$ as in Talangrand's characterization. 
    Let $A:=\{a_n: n \in\omega\}$ be a countable dense subset of $X$ and note that a sequence $\bm{x}\in X^\omega$ is $\nu$-maldistributed if and only if $\Gamma_{\bm{x}}(\mathcal{I})=X$. Taking into account that the set of $\mathcal{I}$-cluster points $\Gamma_{\bm{x}}(\mathcal{I})$ is closed, see e.g. \cite[Lemma 3.1(iv)]{MR3920799}, then $\bm{x}$ is $\nu$-maldistributed if and only if $A\subseteq \Gamma_{\bm{x}}(\mathcal{I})$. Since $A$ is countable and the family of meager subsets of $X$ is a $\sigma$-ideal, it is enough to show that
    $$
    \forall \eta \in X, \quad \mathcal{S}_\eta:=\{\bm{x}\in X^\omega: \eta \in \Gamma_{\bm{x}}(\mathcal{I})\}
    $$
    is comeager. 
    To this aim, 
    fix $\eta \in X$. Consider the Banach–Mazur game defined as follows: Players I and II choose alternatively nonempty open subsets of $X^\omega$ as a nonincreasing chain
    $$
    U_0\supseteq V_0 \supseteq U_1 \supseteq V_1 \supseteq \cdots,
    $$
    where Player I chooses the sets $U_0, U_1,\ldots$; Player II has a winning strategy if $\bigcap_n V_n \subseteq \mathcal{S}_\eta$. It follows by \cite[Theorem 8.33]{MR1321597} that Player II has a winning strategy if and only if $\mathcal{S}_\eta$ is comeager. In fact, let $d$ denote the metric on $X$ and suppose that the nonempty open set $U_n$ has been chosen by Player I. Then $U_n$ contains a nonempty basic open set $B_n$ of $X^\omega$ with support in a subset of the coordinates $\{0,1,\ldots,\kappa_n\}$. Pick $j_n\in\omega$ such that $\min I_{j_n}>\kappa_n$. Then it is enough that Player II chooses 
    $$
    V_n:=\{\bm{x} \in B_n: \forall i \in I_{j_n}, \quad d(x_i,\eta)<2^{-n}\}.
    $$
    This is indeed a winning strategy for Player II: if $\bm{x} \in \bigcap_n V_n$ then for every $\varepsilon>0$ we have that $\{n \in \omega: d(x_n,\eta)<\varepsilon\}$ contains infinitely many intervals $I_k$, hence it is an $\mathcal{I}$-positive set. Therefore $\bm{x} \in \mathcal{S}_\eta$. 
    It follows that $\Sigma_\nu(X)$ is comeager.

    \medskip

    \ref{item:3mainmeager} $\implies$ \ref{item:2mainmeager}. Suppose that $\Sigma_{\nu}(X)$ is comeager and fix $\eta \in X$. Since $X$ is a complete metric space, then $X^\omega$ is Baire and the set $\mathcal{S}_\eta$ defined in \eqref{eq:definitionSeta} is comeager as it is a superset of $\Sigma_{\nu}(X)$. Therefore $\mathcal{I}$ is meager by Proposition \ref{prop:converse}. 
\end{proof}

\begin{rmk}
As it follows from the proof above, the implication \ref{item:2mainmeager} $\implies$ \ref{item:3mainmeager} holds for all separable metric spaces $X$. 
\end{rmk}

\medskip

We conclude with the proof of Proposition \ref{prop:lscsm}. 
\begin{proof}
[Proof of Proposition \ref{prop:lscsm}]
Fix $\alpha \in (0,1)$ and for each $n \in \omega$ define $g_\alpha(n):=\min\{k \in \omega: \varphi([n,n+k])\ge 1-\alpha/4\}$. Now, pick $A\subseteq \omega$ such that $\|\omega\setminus A\|_\varphi \le 1-\alpha$ and fix $n_A \in \omega$ such that $\varphi((\omega\setminus A) \cap [n,\infty)) \le 1-\alpha/2$ for all $n\ge n_A$. Since $\varphi$ is a submeasure, it follows that, for all integers $n\ge n_A$, we get
\begin{displaymath}
    \begin{split}
        \varphi(A \cap [n,n+g_\alpha(n)])&\ge \varphi([n,n+g_\alpha(n)])-\varphi((\omega\setminus A) \cap [n,n+g_\alpha(n)])\\
        &\ge \varphi([n,n+g_\alpha(n)])-\varphi((\omega\setminus A) \cap [n,\infty)) \ge \nicefrac{\alpha}{4},
    \end{split}
\end{displaymath}
hence $A \cap [n,n+g_\alpha(n)]\neq \emptyset$. The second part follows by Theorem \ref{thm:misiktothmain}. 
\end{proof}

\section*{Acknowledgments}

The author is thankful to two anynomous reviewers for a really careful reading of the manuscript which allowed to fix a gap in a previous version of the manuscript. 
The author is also grateful to Marek Balcerzak (Lodz University of Technology, PL) for the suggestion of \cite[Open Problem 5.1]{MR4781067}. 

\bibliographystyle{amsplain}
\providecommand{\href}[2]{#2}

\end{document}